\documentclass[12pt,a4paper]{amsart}  
\input{preamble.sty}
\begin{document}
\title{A $T1$ theorem for weakly singular integral operators}
\author{ Antti V. V\"ah\"akangas}
\address{Department of Mathematics and Statistics,
P.O. Box 68 (Gustaf H\"allstr\"omin katu 2b),
FI-00014 University of Helsinki, Finland}
\email{antti.vahakangas@helsinki.fi}
\date{}

\begin{abstract}
We establish conditions in the spirit of the $T1$ theorem of
David and Journ\'e which guarantee the boundedness
of  $\nabla T$ on $L^p(\R^n)$,
where $T$ is an integral transformation and $1<p<\infty$. These 
 are natural
size and regularity conditions for the kernel of
the integral transformation, along
with the sharp condition $T1,T^t1\in\mathcal{I}^1(\mathrm{BMO})$.
A simple example satisfying these conditions
is the Riesz potential denoted by $\mathcal{I}^1$.
\end{abstract}
\maketitle
%\tableofcontents

\section{Introduction}

The purpose of this paper is to explore boundedness properties of certain weakly singular integral operators, 
and to study their connections to so called almost diagonal potential operators.
Given a linear integral operator
\begin{equation}\label{olii}
Tf(x)=\int_{\R^n} K(x,y)f(y){\rm{d}}y,
\end{equation}
where $|K(x,y)|\le C_K|x-y|^{1-n}$,
one is often interested in conditions under which the operator 
\[
\nabla Tf(x) =\nabla_x \int_{\R^n} K(x,y)f(y){\rm{d}}y
\] has
a bounded extension to $L^p(\R^n)$ or to some
other relevant function space.
To state this more precisely, we assume that
$T:\mathcal{S}_0\to \mathcal{S}'/\mathcal{P}$
is continuous and that
$Tf$ is the regular distribution given by \eqref{olii} for
all  $f\in \mathcal{S}_0$; this is the space of Schwartz
functions with all vanishing moments and $\mathcal{S}'/\mathcal{P}$ is
its topological dual space.
Then we want to know when
$\partial_j T$ has the extension for all $j\in \{1,2,\ldots,n\}$.
%This differentiates
%our approach from \cite{torres} where
%this problem is also addressed.

Weakly singular integral operators emerge in the context 
of pseudodifferential operators with symbols in 
$\dot S^{-1}_{1,1}$
\cite{gtorres}. This
class of homogeneous symbols consists of smooth functions 
$a:\R^n\times(\R^n\setminus\{0\})\to \C$
satisfying
\begin{equation*}
|\partial^\alpha_\xi \partial^\beta_x a(x,\xi)|\le 
C_{\alpha,\beta}|\xi|^{-1-|\alpha|+|\beta|}
\end{equation*}
for all multi-indices $\alpha,\beta\in \N^n_0$.
The Pseudodifferential operator $T_a$ associated with symbol 
$a\in\dot{S}^{-1}_{1,1}$ is
the linear integral operator defined for $f\in\mathcal{S}_0$ 
pointwise by
\[
T_af(x)=\int_{\R^n}Êa(x,\xi)\hat f(\xi)
e^{{i} x\cdot \xi}{\rm{d}}\xi.
\]
It is known that such an operator has an alternative 
weakly singular representation
\eqref{olii} with highly smooth kernel $K_a$. 
On the other hand, boundedness properties of 
$\nabla T_a$ 
are not automatically guaranteed.
Singular integral methods are available
when treating with pseudodifferential operators as 
$\nabla T_a$ is  such an operator when the differentiation is
taken under the integral sign. 
As a simple example, well known Riesz transforms, that are
bounded operators on $L^p$-spaces for $p\in (1,\infty)$,
can be realized in such a form:
\[\mathcal{R}=(\mathcal{R}_1,\ldots,\mathcal{R}_n)=\nabla \mathcal{I}^1.\]
Here $\mathcal{I}^1$ stands for the first order Riesz potential,
which is a weakly singular integral operator associated to the kernel $K(x,y)=c_n|x-y|^{1-n}$.

Our
approach is to treat operators  without first
differentiating them under the integral sign. This approach is also
adopted in \cite{torres}, where certain results about 
boundedness of operators with singular kernels of different sizes are available.
Recent work \cite{t1dom} of  present author  deals directly with boundedness properties of weakly singular integral
operators on Euclidean domains, whereas
the present work is oriented towards 
sharp
connections between weakly singular integral operators and so called almost diagonal operators
of Frazier and Jawerth \cite{F-J1}.

Let us next quantify certain kernel classes and weakly singular integral operators.
%in general
%settings have been studied extensively and
%aspects such as boundedness 
%and convergence are typically addressed 
%\cite{christ,han-sawyer,nazarov,torres}.
%In this paper
%we obtain 
%{\em sharp boundedness results
%under minimal
%smoothness assumptions of the kernel $K$}.
%This differentiates
%our approach from \cite{torres} where
%this problem is also addressed.
A continuous function 
$K:\R^n\times \R^n\setminus \{(x,x)\}\to \C$ is
a standard kernel of order $-1$, if 
$|K(x,y)|\le C_K|x-y|^{1-n}$ for distinct $x,y\in\R^n$, and
there exists
$\delta\in (0,1]$ such that 
\begin{equation}\label{oliv}
\begin{split}
&|K(x+h,y)+K(x-h,y)-2K(x,y)|\\
&\qquad +|K(x,y+h)+K(x,y-h)-2K(x,y)|\le C_K
|h|^{1+\delta}|x-y|^{-n-\delta}
\end{split}
\end{equation}
when $0<|h|\le |x-y|/2$. 
If $T$ as in \eqref{olii} is associated with a 
standard 
kernel $K$ of order $-1$, we denote
$T\in \mathrm{SK}^{-1}(\delta)$,
and we say that $T$ is a weakly singular integral operator (of order $-1$).
 As an example,
$T_a\in\mathrm{SK}^{-1}(1)$ if $a\in \dot S^{-1}_{1,1}$.

Below is  our main result. Formulation
below is a rather concrete one, and we will later give 
formulations based on homogeneous Triebel--Lizorkin spaces.

\begin{thm}
Assume that $T\in\mathrm{SK}^{-1}(\delta)$.
Then the following three conditions 
are equivalent:
\begin{itemize}
\item $T1,T^t1\in \mathcal{I}^1({\mathrm{BMO}})$, where $\mathcal{I}^1$ is 
the
first order Riesz potential;
\item $\nabla T$ and $\nabla T^t$ have bounded
extension to $L^2$;
\item $\nabla T$ and $\nabla T^t$ have bounded
extension to $L^p$ for all $1<p<\infty$.
\end{itemize}
\end{thm}

As another result we characterize operators 
$T\in \mathrm{SK}^{-1}(\delta)$
satisfying the condition $T1=0=T^t1$ by a slightly modified
almost diagonality condition of Frazier and Jawerth \cite{F-J1}.
%This justifies our previous terminology about
%minimal smoothness
%of the kernels.

It is only natural that almost diagonality or almost 
orthogonality 
has a crucial role here. Corresponding decompositions of 
operators
appeared already
in the proof of the classical $T1$ theorem of David
and Journ\'e \cite{D-J}, and have been the standard technique to
deal with non-convolution type singular integral operators
on very general settings ever since 
\cite{fhw,han-sawyer,meyer-c,torres}.

To remind the reader we state the classical $T1$ theorem here.
Let
$T:\mathcal{S}\to \mathcal{S}'$ be a continuous linear operator that is associated 
with a kernel
$K:\R^n\times \R^n\setminus\{(x,x)\}\to \C$ so that 
\eqref{olii} holds for every $x\not\in {\mathrm{supp}} f$. 
%a {\em standard kernel}, if it satisfies 
Assume further that $K$ satisfies
$|K(x,y)|\le C_K|x-y|^{-n}$ for distinct $x,y\in\R^n$, and there is $\delta\in (0,1]$ such that
\[
|K(x+h,y)-K(x,y)|+|K(x,y+h)-K(x,y)|\le C_K|h|^\delta
|x-y|^{-n-\delta}
\]
whenever $0<|h|\le |x-y|/2$.  Then we say that $T$ is
associated with a standard kernel and denote this by 
$T\in\mathrm{SK}(\delta)$. 
Here is the classical $T1$ theorem \cite{D-J}:

\begin{thm}
Let $T\in \mathrm{SK}(\delta)$. Then the following
two conditions are equivalent:
\begin{itemize}
\item $T1,T^t1\in {\mathrm{BMO}}$ and $T$ has a certain weak boundedness 
property;
\item $T$ has a bounded extension to $L^2$ i.e., $T$ is a 
Calder\'on--Zygmund operator.
\end{itemize}
\end{thm}

The main issue here is the $L^2$-boundedness.
The $L^p$-boundedness
for $1<p<\infty$ follows from the $L^2$-boundedness 
and this was known before the celebrated result
of David and Journ\'e. The additional weak boundedness property
is required because the kernel is in some sense too singular
near the diagonal $\{(x,x)\}\subset \R^n\times \R^n$
and therefore does not specify the operator $T$ uniquely.
This undesired phenomenon is ruled out in the context
of standard kernels of order $-1$.

%We have not pursued
%for the most general formulations of our results:
%we do not discuss homogeneous Besov spaces, see 
%\cite{F-J1,triebel1},
%and we limit the range of exponents 
%to $1\le p,q\le \infty$.
%Also, the smoothing index of potential operators
%can be extended
%to the interval $(0,2)$ such that essentially all
%the results with appropriate modifications hold true.
%If the smoothing index belongs to the interval $(0,1)$ it is 
%equivalent
%to use first order differences in the analogue of \eqref{oliv}.
%The classical $T1$ theorem above corresponds
%to the formal limiting case of potential operators with
%smoothing index zero.

Here is the organization of this paper. In
Section 2 we review homogeneous Triebel--Lizorkin
spaces along
with the $\phi$-transform. The almost diagonality
is also 
presented. Section 3 contains definitions of
potential operators and associated kernels.
Section 4 and Section 5 are devoted to the special
$T1$ theorem ($T1=0=T^t1$)
and to  its converse:
such special operators correspond to 
almost diagonal potential
operators on sequence spaces. Section
6 deals with the full $T1$ theorem
($T1,T^t1\in \mathcal{I}^1({\mathrm{BMO}})$).
We prove that this condition is not
only sufficient but also necessary for the boundedness 
properties of $T$.

\section{Preliminaries}

%We shall study linear operators
%on homogeneous
%Triebel--Lizorkin spaces 
%(the same techniques work well in homogeneous Besov
%spaces \cite{F-J1,triebel1} but we restrict ourselves to
%the usually more difficult
%case).  When suitably parametrized, these
%Banach spaces cover many classical spaces 
%such as $L^p$ spaces for $1<p<\infty$ and, more generally, the 
%homogeneous
%Sobolev spaces $\dot W^{\alpha, p}$ for $1<p<\infty$ for
%$\alpha\in \R$. 
%In particular this approach yields a unified approach to 
%the study of
%bounded linear operators between these classical spaces.

%Recall the $\phi$-transform identity of Frazier and Jawerth
%\cite{F-J2}: 
%$f=\sum \langle f,\phi_Q\rangle\psi_Q$,
%where the summation is over all dyadic cubes $Q$ in $\R^n$
%and $\phi_Q,\psi_Q$ are translations and dilatations of 
%simple test functions
%$\phi,\psi$, respectively.
%This identity gives rise
%to sequence spaces
%whose certain subspaces
%are isomorphic to  Triebel--Lizorkin spaces.
%Thus the
%boundedness of a linear operator on Triebel--Lizorkin spaces can be 
%studied by studying the induced operator on the sequence 
%spaces. Frazier and Jawerth have quantified a sufficient condition,
%{\em the almost diagonality\/}, for the
%boundedness of the induced operator \cite{F-J1}.

\subsection{Notation}

We denote $\N_0=\{0,1,2\ldots\}$ and $\N=\{1,2,\ldots\}$.
If $a,b\in \R$, then we denote
$a\vee b=\max\{a,b\}$ and $a\wedge b=\min\{a,b\}$.
We work on the Euclidean space $\R^n$ 
with $n\ge 2$.
We denote $X=X(\R^n)$ and usually omit
the ``$\R^n$" whenever
$X$ is a function space, whose members are defined on
$\R^n$.

Let $\nu\in \Z$, $y\in \R^n$, and $f:\R^n\to \C$. Then 
we denote
$f_\nu(x)=2^{\nu n} f(2^\nu x)$, $\tau_{y} f(x)=f(x-y)$,
and $\tilde f(x)=f(-x)$ for all $x\in \R^n$. 
The second order difference is denoted by
$\Delta_{y} f=\tau_{-y}f+\tau_y f-2 f$.
The Fourier transform is defined by
\[
\hat f(\xi)=\mathcal{F}(f)(\xi)=(2\pi)^{-n/2}\int_{\R^n}Êf(x)
e^{-{i} x\cdot \xi}{\mathrm d}x
\]
whenever $f\in L^1$. 

The Schwartz class
of rapidly decaying smooth functions,
equipped with its usual topology, is denoted by $\mathcal{S}$.
Its topological dual space is the
space of tempered distributions $\mathcal{S}'$.
%Its topology is induced by the countable collection
%\[
%p_N(\phi)=\sup_{|\alpha|\le N} \sup_{x \in \R^n} 
%(1+|x|^2)^N|(\partial^\alpha \phi)(x)|<\infty,
%\quad \text{ for all } N\in \N_0,
%\]
% of seminorms.
Let $\mathcal{S}_0$ be the space of all Schwartz functions
$\phi\in\mathcal{S}$ such that 
$\partial^\alpha \hat \phi(0)=0$ for all
multi-indices $\alpha\in \N_0^n$. Then 
$\mathcal{S}_0$ is a closed subspace
of $\mathcal{S}$. If $\mathcal{S}_0'$ denotes
the topological dual space of $\mathcal{S}_0$, then
$\mathcal{S}_0'$ and $\mathcal{S}'/\mathcal{P}$ are isomorphic
\cite[5.1.2]{triebel1}. 
The Riesz $s$-potential, $s\in \R$, is defined by
\[\mathcal{I}^s(\phi)=\mathcal{F}^{-1}(|\xi|^{-s} \hat \phi).\]
These linear operators map $\mathcal{S}_0$
to itself continuously and 
%  as the following theorem implies
%(we omit the relatively easy proof):
%\begin{thm}\label{rp}
%Let $s\in \R$.
%Then the mapping $\mathcal{I}^s:\mathcal{S}_0\to \mathcal{S}_0$
%is continuous.
%\end{thm}
extend
to $\mathcal{S}'/\mathcal{P}$ by duality:
if $\Lambda\in \mathcal{S}'/\mathcal{P}$ then
$\mathcal{I}^s(\Lambda)\in \mathcal{S}'/\mathcal{P}$, where
$\langle \mathcal{I}^s(\Lambda),\phi\rangle = 
\langle \Lambda,\mathcal{I}^s(\phi)\rangle$.
%This definition  since it 
%holds whenever
%$\Lambda\in \mathcal{S}_0$.
For $0<s<n$ holds the integral formulation
\[
\mathcal{I}^s(\phi)(x) = C_{n,s}\int_{\R^n} 
\frac{\phi(y)}{|x-y|^{n-s}}{\mathrm d}y
\]
%whenever $\phi$ is a measurable function with some 
%additional regularity; especially this holds
if $\phi\in \mathcal{S}_0$.
%Later we recall the fact that 
%Riesz potentials are isomorphisms of
%homogeneous Triebel--Lizorkin spaces.

The parameters $\nu \in \Z$ and $k\in \Z^n$ define a dyadic cube
\[
Q_{\nu k} =\{(x_1,\ldots,x_n)\in \R^n\,:\, k_i\le 2^\nu x_i <
k_i+1\text{ for  } i=1,\ldots, n\}.
\]
Denote the ``lower left-corner'' $\,2^{-\nu}k$ of 
$Q=Q_{\nu k}$ by
$x_Q$ and the sidelength $2^{-\nu}$ by $l(Q)$. 
Denote by $\mathcal{D}$ the set of all dyadic cubes in $\R^n$
and by $\mathcal{D}_\nu\subset \mathcal{D}$ the set of all dyadic 
cubes with sidelength $2^{-\nu}$.
Given a function $f\in L^2$, we
denote
\[
f_Q(x)=|Q|^{-1/2} f(2^\nu x - k) = |Q|^{1/2} f_\nu(x-x_Q).
\]
Notice that  $f_Q$ is $L^2$-normalized because
$||f_Q||_2 = ||f||_2$. 
The $L^2$-normalized characteristic
function of $Q\in\mathcal{D}$ is denoted by
$\tilde \chi_Q=|Q|^{-1/2}\chi_Q$, where $\chi_Q$
is the characteristic function of $Q$ (this overrides the
reflection defined above).

\subsection{Homogeneous Triebel--Lizorkin spaces}

For the convenience of the
reader we provide some details about homogeneous Triebel--Lizorkin spaces. These details
are mostly based on
\cite{F-J1}. For further properties about these
spaces, see the fundamental reference \cite{triebel1}. 

First of all, we need the Littlewood--Paley functions:

\begin{defn}
A function $\phi\in \mathcal{S}_0$ 
is a {\em Littlewood--Paley function\/} if it 
satisfies the following properties:
\begin{itemize}
\item $\hat \phi$ is a real valued radial function;
\item ${\mathrm{supp}} \hat \phi\subset \{\xi\in \R^n\,:\, 1/2< |\xi|< 2\}$;
\item $|\hat\phi(\xi)|\ge c>0$ if $3/5\le |\xi|\le 5/3$.
\end{itemize}
The function $\psi$ defined by
$\hat \psi=\hat\phi/\rho$, 
$\rho(\xi)=\sum_{\nu\in \Z} (\hat\phi(2^\nu \xi))^2$,
is a {\em Littlewood--Paley dual function\/} 
related to $\phi$. It is 
is a Littlewood--Paley function itself and we have
the identity
\[
\sum_{\nu\in \Z} \hat\phi(2^{-\nu} \xi)\hat \psi(2^{-\nu} \xi)
=1,\quad\text{if }\xi\not=0.
\]
\end{defn}

\begin{defn}\label{triebelit}
Fix a Littlewood--Paley function $\phi$.
Let  $1\le p,q\le\infty$ and $\alpha\in \R$. The
{\em homogeneous Triebel--Lizorkin space\/} $\dot F^{\alpha q}_p$ 
is the normed vector space of all 
$f\in \mathcal{S}'/\mathcal{P}$ such that
\[
||f||_{\dot F^{\alpha q}_p} = \bigg|\bigg| \bigg(\sum_{\nu \in \Z} 
\big(2^{\nu \alpha} |\phi_\nu \star f|\big)^q
\bigg)^{1/q}\bigg|\bigg|_{L^p} < \infty,
\]
if $p<\infty$, and
\[
||f||_{\dot F^{\alpha q}_\infty} 
=\sup_{P\in \mathcal{D}}\bigg(\frac{1}{|P|}\int_P 
\sum_{\nu=-\log_2 l(P)}^\infty
\big(2^{\nu\alpha} |\phi_\nu\star f(x)|\big)^q{\rm d}x\bigg)^{1/q}<
\infty.
\]
If $1\le p<q=\infty$, we use $\sup$-norms instead of $l^q$-norms.
Define also 
$\dot F^{\alpha \infty}_\infty=\sup_{\nu\in\Z} 2^{\nu\alpha}
||\phi_\nu\star f||_{L^\infty}$.
\end{defn}

The definition of homogeneous Triebel--Lizorkin
spaces does not depend on the choice of Littlewood--Paley 
function $\phi$;
for a proof, see \cite[Remark 2.6.]{F-J1}.
An easy consequence of this is the following:

\begin{thm}\label{isomet}
Let $\alpha,s\in \R$ and
$1\le p,q\le\infty$. Then $\mathcal{I}^s$ is an isomorphism 
of $\dot F^{\alpha q}_p$
onto $\dot F^{s+\alpha,q}_p$.
\end{thm}

The homogeneous Triebel--Lizorkin spaces cover
many interesting function spaces; here
 $\approx$ indicates that the corresponding norms are 
equivalent:

\begin{itemize}
\item $\dot F^{02}_p \approx L^p$, where $1<p<\infty$;
\item $\dot F^{02}_1 \approx H^1$;
\item $\dot F^{0 2}_\infty \approx {\mathrm{BMO}}$;
\item $\dot F^{\alpha 2}_p \approx \dot W^{\alpha, p}$, where 
$1<p<\infty$.
\end{itemize}

Here $\dot W^{\alpha, p}$ is the homogeneous Sobolev space defined
as the space of all $f\in \mathcal{S}'/\mathcal{P}$ with
$\mathcal{I}^{-\alpha}f\in L^p$. This space is normed with
$||f||_{\dot W^{\alpha, p}}=||\mathcal{I}^{-\alpha}f||_p.$ 
Using the Riesz transforms, we have
$||f||_{\dot W^{1, p}}\approx ||\nabla f||_p$ for all $1<p<\infty$.
Detailed proofs or further references 
of these identifications can be found
in \cite{grafakos}.

%The homogeneous Triebel--Lizorkin space
%$\dot F^{12}_\infty\approx \mathcal{I}^1({\mathrm{BMO}})$
%is of particular interest to us. Let us therefore mention
%the following characterization
%in \cite{str}:

%\begin{thm}
%A function $f$ is in $\mathcal{I}^1({\mathrm{BMO}})$ if and only if
%the difference quotients $(f(x+y)-f(x))/|y|$ are in ${\mathrm{BMO}}$ 
%with ${\mathrm{BMO}}$-norm uniformly bounded in $y\not=0$. In particular,
%Lipschitz functions are contained in $\mathcal{I}^1({\mathrm{BMO}})$.
%\end{thm}

\subsection{Sequence spaces and the $\phi$-transform}

Here we define the sequence spaces
$\dot f^{\alpha q}_p$ and $\phi$-transform following still
 the treatment
in \cite{F-J1}. 
%The homogeneous Triebel--Lizorkin spaces are isomorphic
%to certain subspaces of these sequence spaces; 
%These provide 
%alternative  characterizations
%for various function spaces. 
%For the goals mentioned above, we need the concept of 
%$\phi$-transform and its left inverse. These will
%be based on the $\phi$-transform identity.

%We begin with the definition of the aforementioned sequence spaces:

\begin{defn}\label{sekvenssit}
Let $\alpha\in \R$, $1\le p,q\le \infty$, and
$\tilde \chi_Q=|Q|^{-1/2}\chi_Q$. 
The {\em sequence space\/} $\dot f^{\alpha q}_p$ is the collection
of all complex valued sequences $s=\{s_Q\}_{Q\in \mathcal{D}}$
such that 
\[
||s||_{\dot f^{\alpha q}_p} = \bigg|\bigg| \bigg(\sum_{Q\in
\mathcal{D}} \big(|Q|^{-\alpha/n}|s_Q|
\tilde\chi_Q\big)^q
\bigg)^{1/q}\bigg|\bigg|_{L^p} < \infty,
\]
if $1\le p,q<\infty$. If $1\le p<q=\infty$, we use $\sup$-norm 
instead of $l^q$-norm.
Also,
\[
||s||_{\dot f^{\alpha q}_\infty} = \sup_{P\in \mathcal{D}}
\bigg( \frac{1}{|P|} \sum_{Q\subset P} 
\big(|Q|^{-\alpha/n-1/2+1/q}|s_Q|\big)^q\bigg)^{1/q},
\]
if $1\le q<\infty$.
Finally we define 
$||s||_{\dot f^{\alpha\infty}_\infty} = \sup_{Q\in \mathcal{D}} 
\{|Q|^{-\alpha/n-1/2}|s_Q|\}$.
\end{defn}

\begin{rmk}\label{huhu}
If $1\le p<\infty$ and $s\in \dot f^{\alpha q}_p$, then
the series $\sum_{Q} s_Q\psi_Q$ converges unconditionally
in the weak* topology of $\mathcal{S}'/\mathcal{P}$; for a proof, see 
\cite[pp. 152--154]{Kyr}.
If also $1\le q<\infty$, then sequences with finite support are
dense in $\dot f^{\alpha q}_p$.
\end{rmk}

Then we define the $\phi$-transform and its left inverse. 
%These
%are important operators connecting homogeneous Triebel--Lizorkin
%spaces to the corresponding sequence spaces.

\begin{defn}
For a Littlewood--Paley function $\phi$, the $\phi$-transform
$S_\phi$ is the operator taking each $f\in \mathcal{S}'/
\mathcal{P}$ to
the sequence
$S_\phi f=\{(S_\phi f)_Q=\langle f,\phi_Q\rangle\,:\,Q\in 
\mathcal{D}\}$. Let $\psi=\psi_\phi$ be a dual Littlewood--Paley function
with respect to $\phi$.
Then the {\em left inverse $\phi$-transform\/} $T_\psi$ is the 
operator taking the sequence
$s=\{s_Q\}_{Q\in \mathcal{D}}$ to
$T_\psi s =\sum_Q s_Q\psi_Q$. 
\end{defn}

The technical definition of the 
summation associated with
$T_\psi$ is 
given by 
the following theorem. Its
detailed proof
can be found in \cite{F-J2}.

\begin{thm}\label{phi_ide}
Suppose that $\phi$ and $\psi=\psi_\phi$ is a dual pair of 
Littlewood--Paley functions. Assume that  
$f\in \mathcal{S}'/\mathcal{P}$.
Then for all $\gamma\in \mathcal{S}_0$ we have
\[
\langle f,\gamma\rangle =\sum_{\nu\in \Z}\sum_{k\in \Z^n} 
\langle f,\phi_{Q_{\nu k}}\rangle\langle \psi_{Q_{\nu k}},
\gamma\rangle.
\]
Here the inner sum converges absolutely for any fixed $\nu$.
In particular, we recover the {\em $\phi$-transform identity\/}:
$\id = T_\psi\circ S_\phi$
on $\mathcal{S}'/\mathcal{P}$; that is,
\[
f=\sum_{\nu\in \Z}\sum_{k\in \Z^n} \langle f,\phi_{Q_{\nu k}}
\rangle\psi_{Q_{\nu k}},
\]
with convergence in the weak* topology of
$\mathcal{S}'/\mathcal{P}$.
\end{thm}

Relying on the $\phi$-transform identity 
and on maximal operators,
Frazier and Jawerth prove the following technical result
in \cite[Theorem 2.2., Theorem 5.2.]{F-J1}:

\begin{thm}\label{jatkuv}
Let $\alpha\in \R$ and $1\le p,q\le\infty$. Then
operators $S_\phi:\dot F^{\alpha q}_p\to \dot f^{\alpha q}_p$ and
$T_\psi:\dot f^{\alpha q}_p\to \dot F^{\alpha q}_p$ are
bounded and $T_\psi\circ S_\phi$ is the identity on
$\dot F^{\alpha q}_p$. 
\end{thm}

\subsection{Almost diagonal potential operators}

We continue our recapitulation of \cite{F-J1} and move
to almost diagonal operators. 
{\em But there is a slight modification}\,:
we are interested in potential operators and this is reflected
in the following definitions and results. 
%Given a linear operator $T:\mathcal{S}_0\to \mathcal{S}'/\mathcal{P}$
%its boundedness on Triebel--Lizorkin spaces can be studied by
%discretizing the operator and studying the corresponding matrix 
%operator
%\begin{equation}\label{entryt}
%A=\{\langle T(\psi_P),\phi_Q\rangle\}_{P,Q\in\mathcal{D}}
%\end{equation}
%on the  sequence spaces.
%The almost diagonal condition is a $p$-independent condition for
%the matrix operator  $A$
% to be  bounded. 
%It  depends only on the magnitude of the entries of the matrix
%assuring that 
%they decay sufficiently fast away from the diagonal.
Let $\epsilon>0$, $P,Q\in \mathcal{D}$, and 
denote 
\[
\omega_{P,Q}(\epsilon)=
\frac{(l(P)
\wedge l(Q))^{1+(n+\epsilon)/2}}{(l(P)\vee l(Q))^{(n+\epsilon)/2}}
\bigg( 1 + \frac{|x_P-x_Q|}{l(P)\vee l(Q)}\bigg)^{-(n+\epsilon)}.
\]

\begin{defn}\label{adpmaar}
Let $T:\mathcal{S}_0\to \mathcal{S}'/\mathcal{P}$ be a continuous 
linear operator whose transpose 
$T^t:\mathcal{S}_0\to \mathcal{S}'/\mathcal{P}$ is also continuous; 
the transpose is defined by the dual operation 
$\langle Tf,g\rangle =\langle f,T^tg\rangle$.
Assume also that there exists $\epsilon>0$ such that
\[
\sup \bigg\{ 
\frac{|\langle T(\psi_P),\phi_Q\rangle|}{\omega_{P,Q}(\epsilon)}
\,:\,P,Q\in\mathcal{D}\bigg\}
<\infty.
\]
Then $T$ is an {\em almost diagonal potential operator\/}
and we denote  $T\in \mathrm{ADP}(\epsilon)$.
\end{defn}

\begin{rmk}
Note that $\mathrm{ADP}(\epsilon)$ is a vector
space and that $\mathrm{ADP}(\epsilon)
\subset \mathrm{ADP}(\epsilon')$ 
if $\epsilon\ge\epsilon'$.
Note that, a priori,
any given dual pair of Littlewood--Paley functions $\phi$ and 
$\psi_\phi$ may appear in the definition of 
$\mathrm{ADP}(\epsilon)$.
However,
later it turns out that the union 
$\cup_{\epsilon>0} \mathrm{ADP}(\epsilon)$ does not  depend
on the chosen dual pair of Littlewood--Paley functions.
\end{rmk}

%Consider a matrix $A$ as in \eqref{entryt}
%with $|A_{P,Q}|\le \omega_{P,Q}(\epsilon)$ for every 
%$P,Q\in\mathcal{D}$.
%To get acquainted with this inequality, 
%Next we indicate why $T\in \mathrm{ADP}(\epsilon)$ should be considered 
%as an operator
%of smoothing index one.
%For all $P,Q\in\mathcal{D}$, we have
%\[
%\langle \nabla T(\psi_P),\phi_Q\rangle = -\langle T(\psi_P),
%\nabla \phi_Q\rangle
%=-l(Q)^{-1} \langle T(\psi_P),(\nabla \phi)_Q\rangle.
%\]
%We may as well study the discretized derivative given by
%the
%matrix $A'$, whose entries are $A'_{P,Q}=-l(Q)^{-1}\langle 
%T(\psi_P),\phi_Q\rangle$.
%%This matrix
%satisfies the inequality
%\begin{equation*}\label{motivo}
%%|A'_{P,Q}|\le \frac{(l(P)\wedge l(Q))^{(n+\epsilon)/2}}{(l(P)
%\vee l(Q))^{(n+\epsilon)/2} }
%\bigg( 1 + \frac{|x_P-x_Q|}{l(P)\vee l(Q)}\bigg)^{-(n+\epsilon)},
%\end{equation*}
%which appears as a starting
%point for almost diagonal operators in the work of 
%Frazier and Jawerth \cite[p. 53]{F-J1} who study
%operators with smoothing index zero. These include operators
%that are bounded on $L^p$-spaces for $1<p<\infty$.

Almost diagonal potential operators are bounded
operators $\dot F^{\alpha q}_p\to \dot F^{1+\alpha, q}_p$
for $1\le p,q<\infty$ and $-1\le \alpha\le 0$.
%This follows by proving the boundedness of the induced operator
%on sequence space and then deriving the desired
%conclusion with the aid of $\phi$-transform.
This follows from the following lemma, whose
proof
is a straightforward modification of \cite[pp. 54--55]{F-J1}.

\begin{lem}\label{apuva}
Assume  $-1\le \alpha \le 0$, $1\le p,q< \infty$,
$T\in \mathrm{ADP}(\epsilon)$, and
$f\in \mathcal{S}_0$. 
Define for every dyadic cube $Q\in \mathcal{D}$ 
\[
\big(A(S_\phi f)\big)_Q = \sum_{P\in \mathcal{D}} \langle 
T(\psi_P),\phi_Q\rangle (S_\phi f)_P.
\]
This sum  converges absolutely.
Moreover, the operator $A$ is
the discrete representative of $T$ in the sense that
$T=T_\psi\circ A\circ S_\phi$
and we have the estimate $||A(S_\phi f)||_{\dot f^{1+\alpha, q}_p} 
\le C||S_\phi f||_{\dot f^{\alpha q}_p}$.
\end{lem}

%With 
%the aid of $\phi$-transform it captures
%the idea that  boundedness
%of the discretized operator
%implies the boundedness of the operator itself.
As a consequence, we have the following:

\begin{cor}\label{jatkuus}
Let $-1\le \alpha\le 0$, $1\le p,q<\infty$, and 
$T\in \mathrm{ADP}(\epsilon)$. Then 
the operators $T$ and $T^t$ have unique
bounded extensions $\dot F^{\alpha q}_p\to 
\dot F^{1+\alpha,q}_p$. 
\end{cor}

\section{Potential operators and associated standard kernels}

All the required machinery is now at our disposal and
we begin with the main theme of this paper, namely
with first order potential operators associated with
standard kernels of order $-1$.
First of all, we use the following terminology corresponding
to Calder\'on--Zygmund operators in the Introduction: 

\begin{defn}\label{potopp}
We say that a linear operator
$T:\mathcal{S}_0\to \mathcal{S}'/\mathcal{P}$
is a {\em potential operator\/} if
it has a bounded extension
$\dot W^{\alpha, 2}\to \dot W^{1+\alpha,2}$
for all $-1\le\alpha\le 0$.
\end{defn}

Next we define the classes $\mathrm{SK}^{-1}(\delta)$, already mentioned
in the Introduction. Their 
role resembles 
that of the standard kernel classes $\mathrm{SK}(\delta)$
from the Introduction:

\begin{defn}\label{potop}
Let 
$T:\mathcal{S}_0\to \mathcal{S}'/\mathcal{P}$ be a linear 
operator
such that
\[
Tf(x) =\int_{\R^n} K(x,y)f(y){\rm{d}}y.
\]
Assume that the kernel $K:\R^n\times \R^n\setminus\{(x,x)\}\to \C$ 
is continuous and satisfies the following conditions
for some $0<\delta\le 1$ and whenever $|h|\le |x-y|/2$
\begin{eqnarray}
&&|K(x,y)|\le C_K|x-y|^{1-n},\label{ee}\\
&&|K(x+h,y)+K(x-h,y)-2K(x,y)|\le 
C_K\frac{|h|^{1+\delta}}{|x-y|^{n+\delta}},
\label{tt}\\
&&|K(x,y+h)+K(x,y-h)-2K(x,y)|\le 
C_K\frac{|h|^{1+\delta}}{|x-y|^{n+\delta}}\label{kk}.
\end{eqnarray}
Then $T$ is associated with a
{\em standard kernel $K$ of order $-1$\/}.
We denote this by $T\in \mathrm{SK}^{-1}(\delta)$.
\end{defn}

\begin{rmk}
Note that 
$Tf$ is well defined for  $f\in\mathcal{S}_0$ because of
property \eqref{ee}. Also,
$\mathrm{SK}^{-1}(\delta)$ is a vector space and $\mathrm{SK}^{-1}(\delta)
\subset \mathrm{SK}^{-1}(\delta')$
if $\delta\ge \delta'$.
The transpose $T^t$ of an operator $T\in\mathrm{SK}^{-1}(\delta)$ is 
defined by
$\langle T^t f,g\rangle = \langle f,T g\rangle$. 
It is associated with the kernel
$K^t:(x,y)\mapsto K(y,x)$, which also possesses  the properties
\eqref{ee}--\eqref{kk}. The transpose
satisfies $T^t\in\mathrm{SK}^{-1}(\delta)$.
\end{rmk}

Let us 
also clarify the terminology concerning ``$T1$'':

\begin{defn}\label{t1}
Fix a Littlewood--Paley function $\phi$.
Let $T\in\mathrm{SK}^{-1}(\delta)$,  $f\in \mathcal{S}'/\mathcal{P}$,
and denote $\eta^j(x)=\phi(x/2^j)/\phi(0)$.
Assume that for all dyadic cubes $Q\in \mathcal{D}$ we have
$\lim_{j\to\infty}\langle T(\eta^j),\phi_Q\rangle 
= \langle f,\phi_Q\rangle$
and the same with $\phi$ replaced in the
brackets by the dual function $\psi=\psi_\phi$. Then we denote  
$T1=f$.
\end{defn}

\begin{rmk}
If $T1$ exists, then Theorem \ref{phi_ide}
implies that it is uniquely defined in $\mathcal{S}'/\mathcal{P}$
with respect to a fixed dual pair of Littlewood--Paley functions.
On the other hand, the  definition of $T1$ seems at first sight 
to depend on the Littlewood--Paley
function $\phi$ and its dual function $\psi$.
However, we will not address
this dependence unless 
explicitly needed.  Instead, we prefer to work with a fixed
dual pair of Littlewood--Paley functions.
\end{rmk}

The central question we shall address in this paper is:
{\em Under what conditions $T\in\mathrm{SK}^{-1}(\delta)$ is a potential
operator?}
%Along the way we obtain a few results related to this.
It turns out that this is
equivalent to the requirement $T1,T^t1\in\dot F^{12}_\infty$.

Let us discuss the implications of condition \eqref{ee} occuring 
in the definition of standard kernels of order $-1$.
It also allows the operator  $T\in \mathrm{SK}^{-1}(\delta)$ to inherit 
all the properties of 
Riesz potential that involve no cancellation but size only. 
To be more precise, let $1\le p<q<\infty$
satisfy $1/p-1/q=1/n$. Then it is well known 
\cite[pp. 415--416]{grafakos} that for all $f\in\mathcal{S}_0$
and $x\in \R^n$,
we have 
\begin{equation}\label{maximal}
|Tf(x)|\le C_K\mathcal{I}^1(|f|)(x) \le C_{K,n,p} 
M(f)(x)^{p/q}||f||^{1-p/q}_p.
\end{equation}
This maximal inequality
allows us to conclude
that $T$ and $T^t$ are continuous
$\mathcal{S}_0\to \mathcal{S}'/\mathcal{P}$. 

%It also allows us to
%recover the following corollary of fractional integration,
%see the references above:
%\begin{cor}\label{finteg}
%Let $T\in \mathrm{SK}^{-1}(\delta)$ and $1\le p<q<\infty$ be 
%such that $1/p-1/q=1/n$. Then for $p>1$, we have
%$
%||Tf||_q \le C_{K,n,p}||f||_p
%$
%for all $f\in \mathcal{S}_0$. For $p=1$, we have the weak type 
%$(1,q)$ inequality
%\[
%|\{x\,:\,|Tf(x)|>\lambda\}| \le\bigg( 
%\frac{C_{K,n}||f||_1}{\lambda}\bigg)^q
%\]
%for all $f\in \mathcal{S}_0$ and $\lambda>0$.
%\end{cor}
Another consequence of \eqref{ee} is the following
uniqueness result:

\begin{prop}\label{uniq}
Let $T\in\mathrm{SK}^{-1}(\delta)$ be
associated with standard kernel $K$ of order $-1$.
If $Tf=0$ in $\mathcal{S}'/\mathcal{P}$ for all
$f\in\mathcal{S}_0$, then $K(x,y)=0$ for
every $(x,y)\in \R^n\times \R^n\setminus\{(x,x)\}$.
\end{prop}

\begin{proof}
Let $Q\in\mathcal{D}$. Then 
$T\phi_Q=0$
in $\mathcal{S}'/\mathcal{P}$. 
Therefore
there exists a polynomial
$P_Q$ such that
$T\phi_Q(x)=P_Q(x)$
for almost every $x\in \R^n$.
Applying \eqref{ee} and estimate \eqref{maximal}, we have for
almost every $x\in \R^n$
\[
|P_Q(x)|\le C_{K}\int_{\R^n} \frac{|\phi_Q(y)|}{|x-y|^{n-1}}
{\rm d}y\le C_{K,n,p}M\phi_Q(x)^{(n-1)/n}||\phi_Q||^{1/n}_1\in 
L^{2n}({\rm d}x).
\]
Thus $P_Q\equiv0$. 
Denote 
$A= \cup_{Q\in\mathcal{D}} \{x\,:\,T\phi_Q(x)\not=0\}$. Then
$m_n(A)=0$. Fix $x_A\in \R^n\setminus A$. Then
$T\phi_Q(x_A)=\int_{\R^n} K(x_A,y)\phi_Q(y){\rm{d}}y = 0$
for every $Q\in \mathcal{D}$. Using Theorem \ref{phi_ide}
we see that $y\mapsto K(x_A,y)=0$ in $\mathcal{S}'/\mathcal{P}$.
Using  \eqref{ee} we see that the formula $y\mapsto K(x_A,y)$ 
induces a tempered
distribution; thus $K(x_A,y)=P_{x_A}(y)$ for every
$y\in \R^n\setminus\{x_A\}$, where $P_{x_A}$ is a polynomial 
depending on $x_A$.  
Using \eqref{ee} again, we
have
$|P_{x_A}(y)|\le C_K|x_A-y|^{1-n}$ for every 
$y\in \R^n\setminus\{x_A\}$. Therefore $P_{x_A}\equiv 0$
and  $K(x_A,y)=0$ for every $y\in \R^n\setminus\{x_A\}$.
Because $m_n(A)=0$
and $K$ is continuous, we see that $K\equiv 0$.
\end{proof}

\subsection{Motivation for further conditions}

Recall the homogeneous pseudodifferential operators 
in the
Introduction, 
or see \cite{gtorres} for more details. These operators
are of interest to us
mainly because
$T_a\in \mathrm{SK}^{-1}(1)$ if $a\in \dot{S}^{-1}_{1,1}$;
see \cite[p. 271]{stein:1993}
for the proof of a similar relation.
These operators provide us motivation for further conditions:
we present a construction of 
a symbol $a\in \dot S^{-1}_{1,1}$ 
such 
that $T_a$ has no bounded extension
$\dot F^{-1 2}_2\to \dot F^{02}_2$.
This construction is a 
small modification of \cite[pp. 272--273]{stein:1993}.

Fix a Littlewood--Paley function $\phi$ with
the additional properties 
\[
\mathcal{F}(\phi_{-1})|B(e_1,\epsilon)\equiv 0 
\equiv \mathcal{F}(\phi_{1})|B(e_1,\epsilon)
\]
and  $\hat \phi|B(e_1,\epsilon) \equiv 1$ for
some $\epsilon>0$ (here $e_1$ is the first base vector on 
$\R^n$). 
Define then a symbol as follows
\[
a(x,\xi)=\sum_{\nu\in\Z} 2^{-\nu} 
e^{-{i} 2^{\nu}e_1\cdot x}
\hat \phi(2^{-\nu}\xi).
\]
It is easy to verify that $a\in \dot S^{-1}_{1,1}$
and that
for $f\in \mathcal{S}_0$ we have the representation
\[
T_af(x) = (2\pi)^{-n/2}\sum_{\nu\in \Z} 2^{-\nu}
e^{-{i} 2^\nu e_1\cdot x} (\phi_\nu\star f)(x).
\]

For the proof of the following unboundedness result see the references above. 
Note, however, that
$T_a$ is bounded operator
$\dot F^{0,2}_2\to \dot F^{12}_2$, see
Theorem 1.1 \cite{gtorres}.

\begin{prop}\label{vesimII}
Let $a\in \dot S^{-1}_{1,1}$ be as above. Then
the operator $T_a\in \mathrm{SK}^{-1}(1)$ has no bounded extension 
$\dot F^{-1,2}_2\to \dot F^{02}_2$ and it is not
a potential operator.
\end{prop}

%\begin{proof}
%Let $\hat f$ be a nonzero real-valued Schwartz function
%with ${\mathrm{supp}} \hat f\subset B(0,\epsilon)$ and 
%define $\hat f_k(\xi) = \hat f(\xi-2^k e_1)$ for all $k\in \N$. 
%Then $\mathcal{F} (\phi_\nu) \hat f_k = \delta_{\nu k}\hat f_k$ 
%for all $k \in \N$ and $\nu\in \Z$. Define
%$F_N=\sum_{k=1}^N 2^{k} k^{-1} f_k$ for all $N\in \N$.
%Then we have
%\begin{eqnarray*}
%||F_N||_{\dot F^{-1,2}_2} 
%&=& \bigg|\bigg| \bigg(\sum_{\nu\in \Z} 
%\big(2^{-\nu} |\phi_\nu\star F_N|\big)^2\bigg)^{1/2}\bigg|\bigg|_2
%=\bigg|\bigg| \bigg(\sum_{\nu=1}^N 
%\big(\nu^{-1}|f|\big)^2\bigg)^{1/2}\bigg|\bigg|_2\\
%&=&\bigg|\bigg| \bigg(\sum_{\nu=1}^N \nu^{-2} |f|^2\bigg)^{1/2}
%\bigg|\bigg|_2
%=||f||_2\Bigg( \sum_{\nu=1}^N \nu^{-2}\Bigg)^{1/2}
%\le C ||f||_2<\infty
%\end{eqnarray*}
%with $C$ independent of $N$. On the other hand, we have
%\[
%TF_N = (2\pi)^{-n/2}\sum_{\nu\in \Z} 2^{-\nu}
%e^{-{i} 2^\nu e_1\cdot x} (\phi_\nu\star F_N)
%=\sum_{\nu=1}^N \nu^{-1} f.
%\]
%Thus, we have
%\[
%||TF_N||_{\dot F^{02}_2} 
%= ||f||_{\dot F^{02}_2}
%\sum_{\nu=1}^N \nu^{-1}\xrightarrow{N\to\infty}\infty.
%\]
%The claim follows from these estimates.
%\end{proof}

\section{A special $T1$ theorem}

As indicated above, we need further 
conditions on 
the operator $T\in \mathrm{SK}^{-1}(\delta)$
to ensure that it is a potential operator. 
In this section we pose the cancellation conditions $T1=0=T^t1$;
these are automatically satisfied
by convolution type operators
whose standard kernel, of order $-1$, is of the form $K(x,y)=k(x-y)$ for
some $k:\R^n\to \C$.
Later these conditions are relaxed with the aid of this
special case and paraproducts,  which constitute a prime
example of potential operators.

We prove the special
$T1$ theorem (where $T1=0=T^t1$) using the theory of almost diagonal
potential operators. This involves an almost diagonality estimate
and this is the very essence of this section in the form of
a tedious but elementary computation performed in Lemma \ref{ar1}
and in Theorem \ref{isolause}.
This is a typical almost diagonality estimate; for instance it 
resembles 
the wavelet proof
for the classical $T1$ theorem, see \cite[pp. 51--57]{meyer-c}

\begin{lem}\label{ar1} Let $T\in \mathrm{SK}^{-1}(\delta)$ and
 $Q=Q_{\nu k}$. Then for all $x\in \R^n$, we have
\[
|T(\psi_Q)(x)|+|T^t(\phi_Q)(x)|\le C_{K,\phi,\psi} l(Q) |Q|^{-1/2} 
\bigg(1+\frac{|x-x_Q|}{l(Q)}\bigg)^{-(n+\delta)}.
\]
\end{lem}

\begin{proof}
Let us prove the following inequalities
\begin{eqnarray}
&&|T(\psi_Q)(x)|\le C_{K,\psi} l(Q)|Q|^{-1/2} \label{vaite2},\\
&&|T(\psi_Q)(x)|\le C_{K,\psi} l(Q)|Q|^{-1/2}\big(l(Q)^{-1}|x-x_Q|
\big)^{-(n+\delta)} \label{vaite1}.
\end{eqnarray}
The claim concerning $T(\psi_Q)$ follows from the inequalities above. 
Similar proof shows that the estimate holds also for 
$T^t(\phi_Q)$.
Let us first prove the inequality \eqref{vaite2};
using the
 maximal inequality \eqref{maximal} with $p=1$ it is obtained
 as follows
\begin{eqnarray*}
|T(\psi_Q)(x)| &\le& 
 C_K M(\psi_Q)(x)^{(n-1)/n}||\psi_Q||_1^{1/n}\\ 
&\le& C_K |Q|^{-(n-1)/2n + 1/2n}||\psi||_\infty^{(n-1)/n} 
||\psi||_1^{1/n}
=C_{K,\psi} l(Q) |Q|^{-1/2}.
\end{eqnarray*}
Then we prove the inequality \eqref{vaite1}.
We can assume that $x\not=x_Q$.
Since $\psi$ is radial and its integral vanishes, we obtain
\begin{eqnarray*}
&&|T(\psi_Q)(x)| = |Q|^{-1/2}\bigg|\int K(x,y)\psi(2^{\nu}y-
2^\nu x_Q){\rm{d}}y\bigg|\\
&&=|Q|^{-1/2} \bigg|\int K(x,x_Q +\omega)\psi(2^\nu \omega)
{\rm d}\omega\bigg|\\
&&\le |Q|^{-1/2}\bigg|\int \big(K(x,x_Q +\omega) 
+ K(x,x_Q-\omega) -2K(x,x_Q)\big)
\psi(2^\nu \omega){\rm d}\omega\bigg|\\
&&\le |Q|^{-1/2}\int \big|K(x,x_Q +\omega) + K(x,x_Q-\omega) 
-2K(x,x_Q)\big|
|\psi(2^\nu \omega)|{\rm d}\omega.
\end{eqnarray*}
Denote $\sigma=x-x_Q$ and
$A=\{\omega\,:\,|\omega|\le |\sigma|/2\}$, 
$B=\{\omega\,:\,|\sigma-\omega | < |\sigma|/2\}$, 
$C=\{\omega\,:\,|\sigma+\omega | < |\sigma|/2\}$,
$D=\R^n\setminus(A\cup B\cup C)$. 
It suffices to show the desired upper bound
while integrating with respect to each of the sets
above.\\
{\em Set $A$.} Using the property \eqref{kk}, we obtain
\begin{eqnarray*}
&&\int_A \big|K(x,x_Q +\omega) + K(x,x_Q-\omega) -2K(x,x_Q)\big|
|\psi(2^\nu \omega)|{\rm d}\omega\\
&&\le C_K |\sigma|^{-(n+\delta)} \int_{\R^n} 
|\omega|^{1+\delta} |\psi(2^\nu \omega)|{\rm d}\omega\\
&&\le C_K |\sigma|^{-(n+\delta)}2^{-\nu(n+1+\delta)} \int_{\R^n}
|\rho|^{1+\delta}|\psi(\rho)|{\rm d}\rho \\
&&\le C_{K,\psi} l(Q)\big(l(Q)^{-1} |x-x_Q|)^{-(n+\delta)}.
\end{eqnarray*}
{\em Set $B$.} 
\begin{eqnarray*}
 &&\int_B \big|K(x,x_Q +\omega) + K(x,x_Q-\omega) -2K(x,x_Q)
\big|
|\psi(2^\nu \omega)|{\rm d}\omega\\
&&\le \int_B \big(|K(x,x_Q +\omega)| + |K(x,x_Q-\omega)| + 
|2K(x,x_Q)|\big)
|\psi(2^\nu\omega)|{\rm d}\omega.
\end{eqnarray*}
Two latter summands are dealt with like the set $D$.
Therefore it suffices to estimate the first summand
using the property \eqref{ee}
as follows
\begin{eqnarray*}
&&\int_B |K(x,x_Q +\omega)| |\psi(2^\nu\omega)|{\rm d}\omega\\
&&\le C_{K,\psi} \int_B |\sigma-\omega|^{1-n} 
|2^\nu \omega|^{-(n+1+\delta)}{\rm d}\omega\\
&&\le C_{K,\psi} 2^{-\nu(n+1+\delta)}|\sigma|^{-(n+1+\delta)}
\int_B |\sigma-\omega|^{1-n}{\rm d}\omega\\
&&= C_{K,\psi} 2^{-\nu(n+1+\delta)}|\sigma|^{-(n+\delta)}
=C_{K,\psi} l(Q)\big(l(Q)^{-1} |x-x_Q|)^{-(n+\delta)}.
\end{eqnarray*}
{\em Set  $C$}. This
is estimated as the integral over the set $B$.\\
{\em Set $D$}. 
Using the
property \eqref{ee}, we obtain
\begin{eqnarray*}
&&\int_D \big|K(x,x_Q +\omega) + K(x,x_Q-\omega) -2K(x,x_Q)\big|
|\psi(2^\nu \omega)|{\rm d}\omega\\
&&\le C_K \int_D |\sigma|^{1-n} |\psi(2^\nu\omega)|
{\rm d}\omega\\
&&= C_K |\sigma|^{1-n}\int_D |\omega|^{-(1+\delta)}
|\omega|^{1+\delta}|\psi(2^\nu \omega)|{\rm d}\omega\\
&&\le C_K |\sigma|^{-(n+\delta)}\int_{\R^n} |\omega
|^{1+\delta}|\psi(2^\nu\omega)|{\rm d}\omega\\
&&\le C_{K,\psi} l(Q)\big(l(Q)^{-1} |x-x_Q|)^{-(n+\delta)}.
\end{eqnarray*}
The last inequality is dealt with like the set $A$.
The estimate \eqref{vaite1} follows.
\end{proof}

\begin{rmk}
Let us now make the following remark regarding the Definition 
\ref{t1}
of $T1$.
Assume that $T\in\mathrm{SK}^{-1}(\delta)$. Then
transposing 
and using Lemma \ref{ar1}, we get the identities
$\lim_{j\to\infty} \langle T(\eta^j),\phi_Q\rangle 
= \int T^t(\phi_Q)$
and similarly for $\psi_Q$'s. 
However, in many cases
it is easier to work with this regularization given
by the family  $\{\eta^j\,:\,j\in \N\}$
 and therefore
we have embedded it already in the definition of $T1$.
\end{rmk}

The following lemma is a special case of 
\cite[Lemma B.1]{F-J1}.
 Note the small
misprint in the original formulation where the inequality 
``$j\ge k$'' should
be replaced by the inequality
``$j\le k$''.

\begin{lem}\label{ar2}
Let $\nu,\mu\in \Z$, $\nu \le \mu$, $x_1\in \R^n$, $\delta>0$, 
and $g,h\in L^1$ be such 
that 
\begin{eqnarray}
&&|g(x)|\le C_g 2^{\nu n/2}(1+2^\nu|x|)^{-(n+\delta)}, \label{e1}\\
&&|g(x)-g(y)|\le C_g 2^{\nu(n/2+\delta/2)}|x-y|^{\delta/2}
\sup_{|z|\le |x-y|}(1+2^\nu|z-x|)^{-(n+\delta)}, \label{e2}\\
&&|h(x)|\le C_h 2^{\mu n/2}(1+2^{\mu} |x-x_1|)^{-(n+\delta)},
\label{e3}\\
&&\int_{\R^n} h(x){\rm{d}}x = 0.\label{e4}
\end{eqnarray}
Then
$|g\star h(x)|\le C_{g,h,\delta}   2^{-(\mu-\nu)(n/2+\delta/2)} 
(1+ 2^\nu |x-x_1|)^{-(n+\delta)}$.
\end{lem}

The following special $T1$ theorem
is one of our main results.
We make a remark concerning the vanishing integrals assumption
after the proof.

\begin{thm}\label{isolause}
Let $T\in \mathrm{SK}^{-1}(\delta)$ be such that
for every dyadic cube
$Q\in \mathcal{D}$ we have
\[
\int_{\R^n} T(\psi_Q)=\int_{\R^n} T^t(\phi_Q)=0.
\]
Then $T\in \mathrm{ADP}(\delta)$; in particular, 
the integral operators $T$ and $T^t$ have bounded extensions
$\dot F^{\alpha q}_p\to \dot F^{1+\alpha,q}_p$
for all $1\le p,q< \infty$ and $-1\le \alpha\le 0$.
\end{thm}

\begin{proof}
The 
boundedness results follow from Corollary \ref{jatkuus} after 
we have
verified the claim concerning the  almost diagonality.
For this purpose let $P=P_{\mu j}$ and $Q=Q_{\nu k}$.
We do a case study with respect to the relative order of
$l(P)$ and $l(Q)$.

{\em Case $l(P)=2^{-\mu} \le 2^{-\nu}=l(Q)$.} 
It suffices to show that
\[
|\langle T(\psi_P),\phi_Q\rangle|\le C_{T,\phi,\psi} 
\frac{l(P)^{1+(n+\delta)/2}}{l(Q)^{(n+\delta)/2}}
\bigg(1+\frac{|x_P-x_Q|}{l(Q)}\bigg)^{-(n+\delta)}.
\] 
Applying Lemma \ref{ar1},
we obtain
\[
|T(\psi_P)(x)|\le C_{K,\psi} l(P) 2^{\mu n/2} 
(1+2^\mu|x-x_P|)^{-(n+\delta)}.
\]
Denote $h(x)=l(P)^{-1}T(\psi_P)(x)$ and $g(x)=\phi_Q(x_Q-x)$.
Then
\[
|g(x)| = |Q|^{1/2}|\phi_\nu(-x)|=2^{-\nu n/2}2^{\nu n}|
\phi(-2^\nu x)|
\le C_\phi 2^{\nu n/2}(1+2^\nu |x|)^{-(n+\delta)},
\]
and also, since $0<\delta\le 1$, we have
\begin{eqnarray*}
|g(x)-g(y)|&=&2^{\nu n/2}|\phi(-2^\nu x)-\phi(-2^\nu y)|\\
&\le& C_\phi 2^{\nu n/2}|2^{\nu}(x-y)|^{\delta/2}\sup_{|z|\le 
2^{\nu}|x-y|} (1+|z-2^\nu x|)^{-(n+\delta)}\\
&=& C_\phi 2^{\nu (n/2+\delta/2)}|x-y|^{\delta/2}\sup_{|\omega|
\le |x-y|} (1+2^\nu|\omega-x|)^{-(n+\delta)}.
\end{eqnarray*}
Using these estimates and Lemma \ref{ar2}, we obtain
\begin{eqnarray*}
|\langle T(\psi_P),\phi_Q\rangle| &=& \bigg|\int T(\psi_P)(x)
\phi_Q(x){\rm{d}}x\bigg|\\
&=&l(P)\bigg|\int g(x_Q-x)h(x){\rm{d}}x\bigg|\\
&=& l(P)|g\star h(x_Q)|\\
&\le& C_{K,\phi,\psi,\delta} l(P)2^{-(\mu-\nu)(n/2+\delta/2)}
(1+2^\nu|x_Q-x_P|)^{-(n+\delta)}\\
&=&C_{K,\phi,\psi,\delta} 
\frac{l(P)^{1+(n+\delta)/2}}{l(Q)^{(n+\delta)/2}}
\bigg(1+\frac{|x_P-x_Q|}{l(Q)}\bigg)^{-(n+\delta)}.
\end{eqnarray*}
This is the desired estimate in this case.
{\em Case $l(P)=2^{-\mu}> 2^{-\nu}=l(Q)$.} 
It suffices to show that
\[
|\langle \psi_P,T^t(\phi_Q)\rangle|=|\langle T(\psi_P), 
\phi_Q\rangle| \le C_{T,\phi,\psi} \frac{l(Q)^{1+(n+\delta)/2}}
{l(P)^{(n+\delta)/2}}
\bigg(1+\frac{|x_P-x_Q|}{l(P)}\bigg)^{-(n+\delta)}.
\]
Computations here are similar to the first case and
we omit the details.
\end{proof}

\begin{rmk}
If  $T$ satisfies
the assumptions of Theorem \ref{isolause}, then 
$T1=0=T^t1$. This is true
because Theorem \ref{isolause} implies
that $Tf,T^tf\in  \dot F^{02}_1\approx H^1$ for every
$f\in\mathcal{S}_0 \subset \dot F^{-12}_1$
and the integrals of $H^1$-functions vanish.
The converse is also true: if
$T\in\mathrm{SK}^{-1}(\delta)$ is such that
$T1=0=T^t1$, then  $T$ satisfies all the 
assumptions of Theorem \ref{isolause}. It is also
worth noting that the
property $T1=0=T^t1$ for $T\in\mathrm{SK}^{-1}(\delta)$ is independent
of the dual pair of Littlewood--Paley functions; this also
follows from the  boundedness properties
as above. 
\end{rmk}
 
We conclude this section with another formulation
of Theorem \ref{isolause}. Its proof uses 
a duality result from
\cite[pp. 76--79]{F-J1} stating that
$(\dot f^{\alpha q}_p)' \approx \dot f^{-\alpha q'}_{p'}$ for
$1\le p,q<\infty$ and $\alpha\in \R$.

\begin{cor}
Let $T\in\mathrm{SK}^{-1}(\delta)$ be such that
$T1=0=T^t1$. Then 
$T$ has an extension as a bounded
operator 
$\dot F^{\alpha q}_p\to \dot F^{1+\alpha,q}_p$
for all $-1\le \alpha\le 0$ and
$1\le p,q\le\infty$ excluding the cases
$(p,q)=(1,\infty)$ and $(p,q)=(\infty,1)$.
\end{cor}

\section{A converse to the special $T1$ theorem}\label{pkt}

Here we prove a converse result to the special $T1$ theorem
by
adapting computations performed in
\cite{hani} to our context of potential operators.
To be more specific, we show that if $T\in\mathrm{ADP}(\epsilon)$,
then
$T\in\mathrm{SK}^{-1}(\delta)$ for suitable $\delta$ and 
$T1=0=T^t1$ (see Definition \ref{t1}).
Combining the results above we find that $T\in\mathrm{ADP}(\epsilon)$
if and only if $T\in\mathrm{SK}^{-1}(\delta)$ and $T1=0=T^t1$.
Thus the role of almost diagonal potential operators
is twofold: first they are used to establish the special
$T1$ theorem and second they are naturally occuring
examples of operators associated with standard kernels of order $-1$.

From the $\mathrm{ADP}$
point of view the characterization above states that  
$\cup_{\epsilon>0} \mathrm{ADP}(\epsilon)$ is independent of the
dual pair of Littlewood--Paley functions. Here it becomes crucial
that the property $T1=0=T^t1$ for $T\in\mathrm{SK}^{-1}(\delta)$ is
independent of the dual pair of Littlewood--Paley functions.

First we recall some notation.
Let $\beta,\gamma>0$,
$P\in\mathcal{D}_\mu$, and $Q\in\mathcal{D}_\nu$. Then we denote
\[
W_{P,Q}(\beta,\gamma)
=2^{-|\nu-\mu|(n+\gamma)/2}\bigg(1+\frac{|x_Q-x_P|}{l(Q)
\vee l(P)}\bigg)^{-(n+\beta)}.
\]
Note that $\omega_{P,Q}(\epsilon)=\big(l(P)\wedge l(Q)\big)W_{P,Q}
(\epsilon,\epsilon)$; recall Definition \ref{adpmaar}.

The following lemma, yielding a control to the matrix
product of two almost diagonal matrices, is a crucial tool here. 
Its proof can be found in \cite[p. 158]{F-J1}.

\begin{lem}\label{apu1}
Let $P,Q\in\mathcal{D}$.
Let $\beta,\gamma_1,\gamma_2>0$, $\gamma_1\not=\gamma_2$,
and $\gamma_1+\gamma_2>2\beta$. Then
\[
\sum_{R\in \mathcal{D}} W_{P,R}(\beta,\gamma_1)
W_{R,Q}(\beta,\gamma_2)\le C_{n,\beta,\gamma_1,\gamma_2}
W_{P,Q}(\beta,\gamma_1\wedge \gamma_2).
\]
\end{lem}

We also use the following homogeneity estimate for a  double sum.
The reader will find it easy to construct the omitted proof.

\begin{lem}\label{apu3}
Let $\beta>\alpha>0$ and $\lambda,\epsilon>0$. Then
\[
\sum_{\nu \in \Z}\sum_{\mu \in \Z}  2^{-|\nu-\mu|\epsilon}
2^{(\mu\wedge \nu)\alpha}
\big(1+(2^\mu\wedge 2^\nu)\lambda \big)^{-\beta}\le 
C_{\alpha,\beta,\epsilon} \lambda^{-\alpha}.
\]
\end{lem}

We are ready for our converse result; recall  Definitions
\ref{potop} and \ref{adpmaar}.

\begin{thm}\label{kernel_theo}
Suppose that the operator 
$T:\mathcal{S}_0\to \mathcal{S}'/\mathcal{P}$
is an almost diagonal potential operator in the sense of 
Definition \ref{adpmaar}; that is, 
$T\in\mathrm{ADP}(\epsilon)$. Then $T$ is 
associated with a standard kernel of order $-1$
so that for every $f\in\mathcal{S}_0$, we have
\[
Tf(x)=\int_{\R^n} K(x,y)f(y){\rm{d}}y.
\]
Here $K:\R^n\times \R^n\setminus\{(x,x)\}\to \C$ is continuous
and satisfies \eqref{ee}--\eqref{kk} with any $\delta>0$ such
that $\delta<\epsilon/2$ and $\delta\le 1$; that is,
$T\in\mathrm{SK}^{-1}(\delta)$. Also, we have
$T1=0=T^t1$ in the sense of Definition \ref{t1}.
\end{thm}

\begin{proof}
Let $x,y\in \R^n$ and denote
\[
K(x,y)=\sum_{Q \in \mathcal{D}} \sum_{P\in \mathcal{D}}
\langle T(\psi_P),\phi_Q\rangle \phi_P(y) \psi_Q(x).
\]
We first prove the estimate \eqref{ee} for this kernel. This
proof will also show that the series above
converges absolutely for any $x\not=y$.
For $x,y\in \R^n$ with $x\not=y$, we have
\begin{equation}\label{size}
|K(x,y)| \le C_T\sum_{Q \in \mathcal{D}} \sum_{P\in \mathcal{D}}
\big(l(P)\wedge l(Q)\big)W_{Q,P}(\epsilon,\epsilon) |\phi_P(y)| 
|\psi_Q(x)|.
\end{equation}
Let $P\in\mathcal{D}_\mu$ and 
 $R\in \mathcal{D}_\mu$ be such that  $y\in R$.
Using  $|\phi_P(y)|\le C_{\phi,\epsilon} 2^{\mu n/2}
(1+\frac{|x_P-y|}{l(P)})^{-n-\epsilon}$, we get
\[
|\phi_P(y)|\le C_{\phi,\epsilon} 2^{\mu n/2}
\bigg(1+\frac{|x_P-x_R|}{l(P)\vee l(R)}\bigg)^{-n-\epsilon}=
C_{\phi,\epsilon}2^{\mu n/2}W_{P,R}(\epsilon,1+\epsilon).
\]
Using the estimates above and Lemma \ref{apu1}, we obtain for any
given $Q\in \mathcal{D}_\nu$
and fixed $\mu\in\Z$ the estimates
\begin{eqnarray*}
\sum_{P\in\mathcal{D}_\mu} W_{Q,P}(\epsilon,\epsilon)|\phi_P(y)|
&\le&
C_{\phi,\epsilon} 2^{\mu n/2} \sum_{P\in\mathcal{D}} W_{Q,P}
(\epsilon,\epsilon)W_{P,R}(\epsilon,1+\epsilon)\\
&\le& C_{\phi,\epsilon} 2^{\mu n/2} W_{Q,R}(\epsilon,\epsilon).
\end{eqnarray*}
Let $Q\in\mathcal{D}_\nu$
and $S\in\mathcal{D}_\nu$ be such that
$x\in S$; computing as above, we find that
$|\psi_Q(x)|\le C_{\psi,\epsilon} 2^{\nu n/2}W_{S,Q}
(\epsilon,1+\epsilon)$ and then that
\begin{eqnarray*}
2^{\mu n/2}\sum_{Q\in\mathcal{D}_\nu} |\psi_Q(x)|
W_{Q,R}(\epsilon,\epsilon)
%&\le& C_{\phi,\psi,\epsilon}2^{\mu n/2+\nu n/2}
%\sum_{Q\in \mathcal{D}} W_{S,Q}(\epsilon,1+\epsilon)W_{Q,R}
%(\epsilon,\epsilon)\\
&\le& C_{\phi,\psi,\epsilon} 2^{\mu n/2+\nu n/2} 
W_{S,R}(\epsilon,\epsilon).
\end{eqnarray*}
Write $\sum_{P\in \mathcal{D}}\sum_{Q\in \mathcal{D}} 
=\sum_{\nu\in \Z}\sum_{\mu\in \Z}\sum_{Q\in\mathcal{D}_\nu}
\sum_{P\in\mathcal{D}_\mu}$
and 
$l(P)\wedge l(Q)=2^{-(\mu\vee\nu)}$ in \eqref{size}. Using
the two estimates above and the
identities $\mu\wedge\nu\le \mu\vee\nu$ and
$\mu+\nu-|\nu-\mu|=2(\mu\wedge \nu)$, we arrive at
\begin{eqnarray*}
|K(x,y)|&\le& C_{T,\phi,\psi,\epsilon}\sum_{\nu \in \Z}
\sum_{\mu \in \Z}  2^{-(\mu\vee \nu)+\mu n/2+\nu n/2} 
W_{S,R}(\epsilon,\epsilon)\\
&=&C_{T,\phi,\psi,\epsilon}\sum_{\nu \in \Z}\sum_{\mu \in \Z}  
2^{-(\mu\vee \nu)+\mu n/2+\nu n/2}2^{-|\nu-\mu|(n+\epsilon)/2}
\bigg(1+\frac{|x_S-x_R|}{l(S)\vee l(R)}\bigg)^{-n-\epsilon}\\
&\le& C_{T,\phi,\psi,\epsilon}\sum_{\nu \in \Z}\sum_{\mu \in \Z}  
2^{-|\nu-\mu|\epsilon/2}2^{(\mu\wedge \nu)(n-1)}
\big(1+(2^\mu\wedge 2^\nu)|x-y|\big)^{-n-\epsilon}.
\end{eqnarray*}
Note that $n+\epsilon>n-1>0$ and $\epsilon/2>0$; therefore using
Lemma \ref{apu3}, we get
$|K(x,y)|\le C_{T,\phi,\psi,\epsilon}|x-y|^{1-n}$.
This is the estimate \eqref{ee}.

Then we  prove \eqref{kk} for $K$;
fix $x,y,h\in \R^n$ such that $x\not=y$ and
$|h|\le |x-y|/2$. Then
\begin{multline}\label{difference}
|K(x,y+h)+K(x,y-h)-2K(x,y)| \\ \le C_T \sum_{Q \in \mathcal{D}} 
\sum_{P\in \mathcal{D}}
\big(l(P)\wedge l(Q)\big)W_{Q,P}(\epsilon,\epsilon) 
|(\Delta_h \phi_P)(y)| |\psi_Q(x)|.
\end{multline}
{\em Assume first that $P\in\mathcal{D}_\mu$ is such that $|h|\le 2^{-\mu}=l(P)$.} If
$R\in\mathcal{D}_\mu$ is such that $y\in R$, then a short computation
shows that
\[
|(\Delta_h\phi_P)(y)|\le
C_{\phi,\epsilon}2^{\mu(n/2+1+\delta)}|h|^{1+\delta}W_{P,R}
(\epsilon,1+\epsilon).
\]
%then
% we find 
%\begin{eqnarray*}
%|(\Delta_h \phi_P)(y)|  
%&\le& C_{\phi,\epsilon} 2^{\mu(n/2+1+\delta)}|h|^{1+\delta}
%\sup_{|w|\le |h|}(1+2^\mu|w+y-x_P|)^{-n-\epsilon}\\
%&\le& C_{\phi,\epsilon} 2^{\mu(n/2+1+\delta)}|h|^{1+\delta}(1+2^
%\mu|x_P-y|)^{-n-\epsilon},
%\end{eqnarray*}
%because 
%$1+\delta\le 2$ and $1+2^\mu|x_P-y|\le 2(1+2^\mu|w+y-x_P|)$ 
%for all $w\in \R^n$ with $|w|\le |h|$. 
%Let $R\in\mathcal{D}_\mu$ be such that $y\in R$. Then  we obtain
%\begin{eqnarray*}
%|(\Delta_h\phi_P)(y)|&\le& C_{\phi,\epsilon} 
%2^{\mu(n/2+1+\delta)}|h|^{1+\delta}
%(1+2^\mu|x_P-x_R|)^{-n-\epsilon}\\
%&=& C_{\phi,\epsilon}2^{\mu(n/2+1+\delta)}|h|^{1+\delta}W_{P,R}
%(\epsilon,1+\epsilon).
%\end{eqnarray*}
Let $S\in\mathcal{D}_\nu$ be such that $x\in S$; then
computing as before and using Lemma \ref{apu3} with assumption
$2\delta<\epsilon$, we find that
\begin{eqnarray*}
&&\sum_{\nu\in\Z}\sum_{\mu:|h|\le 2^{-\mu}}
\sum_{Q\in\mathcal{D}_\nu}\sum_{P\in\mathcal{D}_\mu}
\big(l(P)\wedge l(Q)\big)W_{Q,P}(\epsilon,\epsilon) 
|(\Delta_h\phi_P)(y)| |\psi_Q(x)|\\
&&\le C_{T,\phi,\psi,\epsilon}|h|^{1+\delta}\sum_{\nu \in \Z}
\sum_{\mu \in \Z}  2^{-(\mu\vee \nu)+\nu n/2+\mu n/2+\mu+\mu\delta}
W_{S,R}(\epsilon,\epsilon)\\
%&&\le C_{T,\phi,\psi,\epsilon}|h|^{1+\delta}\sum_{\nu \in \Z}
%\sum_{\mu \in \Z}  2^{-|\nu-\mu|\epsilon/2}2^{\mu\delta}
%2^{(\mu\wedge \nu)n}
%\big(1+(2^\mu\wedge 2^\nu)|x-y|\big)^{-n-\epsilon}\\
&&\le C_{T,\phi,\psi,\epsilon}|h|^{1+\delta}\sum_{\nu \in \Z}
\sum_
{\mu \in \Z}  2^{-|\nu-\mu|(\epsilon-2\delta)/2}
2^{(\mu\wedge \nu)(n+\delta)}
\big(1+(2^\mu\wedge 2^\nu)|x-y|\big)^{-n-\epsilon}\\
&&\le C_{T,\phi,\psi,\epsilon,\delta}
|h|^{1+\delta}|x-y|^{-n-\delta}.
\end{eqnarray*}
{\em Assume then that $P\in \mathcal{D}_\mu$ is such that
$|h|>2^{-\mu}=l(P)$};  then $1< (2^\mu |h|)^{1+\delta}$ and 
therefore
\begin{eqnarray*}
|(\Delta_h \phi_P)(y)|
%&\le& C_{\phi,\epsilon} 2^{\mu n/2} 
%(1+2^{\mu}|y+h-x_P|)^{-n-\epsilon}\\
%&&+C_{\phi,\epsilon} 2^{\mu n/2} 
%(1+2^{\mu}|y-h-x_P|)^{-n-\epsilon}\\
%&&+C_{\phi,\epsilon} 
%2^{\mu n/2} (1+2^{\mu}|y-x_P|)^{-n-\epsilon}\\
&\le& C_{\phi,\epsilon}2^{\mu(n/2+1+\delta)}|h|^{1+\delta} 
(1+2^{\mu}|y+h-x_P|)^{-n-\epsilon}\\
&&+C_{\phi,\epsilon} 
2^{\mu(n/2+1+\delta)}|h|^{1+\delta}(1+2^{\mu}
|y-h-x_P|)^{-n-\epsilon}\\
&&+C_{\phi,\epsilon} 2^{\mu(n/2+1+\delta)}
|h|^{1+\delta}(1+2^{\mu}|y-x_P|)^{-n-\epsilon}.
\end{eqnarray*}
Computing further as in the case $|h|\le l(P)$ we get
\begin{eqnarray*}
&&\sum_{\nu\in\Z}\sum_{\mu:|h|>2^{-\mu}}
\sum_{Q\in\mathcal{D}_\nu}\sum_{P\in\mathcal{D}_\mu}
\omega_{QP}(\epsilon,\epsilon) |(\Delta_h\phi_P)(y)| 
|\psi_Q(x)|\\
&&\le  C_{T,\phi,\psi,\epsilon,\delta}|h|^{1+\delta}
\big(|x-y-h|^{-n-\delta}+|x-y+h|^{-n-\delta}+
|x-y|^{-n-\delta}\big)\\
&&\le  C_{T,\phi,\psi,\epsilon,\delta}
|h|^{1+\delta}|x-y|^{-n-\delta},
\end{eqnarray*}
because $|h|\le |x-y|/2$.
Inserting the above estimates in \eqref{difference}, we have
\[
|K(x,y+h)+K(x,y-h)-2K(x,y)| \le  
C_{T,\phi,\psi,\epsilon,\delta}|h|^{1+\delta}|x-y|^{-n-\delta}.
\]
This concludes the proof of \eqref{kk}. 
The proof of \eqref{tt} and for
the continuity of the kernel
is similar to this; in the latter  use first order differences.

It remains to show that operator $T$ is associated
with the standard kernel $K$ of order $-1$,
and $T1=0=T^t1$. Fix $f,g\in\mathcal{S}_0$. Then
using Theorem \ref{phi_ide} and Lemma \ref{apuva},
we obtain
\[
\langle Tf,g\rangle = 
\sum_{Q \in \mathcal{D}} \sum_{P\in \mathcal{D}}
\langle T(\psi_P),\phi_Q\rangle \langle \phi_P,
f\rangle\langle \psi_Q,
g\rangle.
\]
Using dominated convergence theorem 
with \eqref{maximal} and above kernel size estimate
beginning from \eqref{size}, 
we have
\[
\langle Tf,g\rangle = \int_{\R^n}\int_{\R^n}ÊK(x,y)f(y)g(x) 
{\rm d}y{\rm d}x.
\]
Thus $T$ is associated with the standard kernel $K$ of order $-1$. The 
conclusion
$T1=0=T^t1$ follows using Corollary \ref{jatkuus}
and the properties of $\dot F^{02}_1\approx H^1$.
\end{proof}

\begin{rmk}
Assume that $T\in\mathrm{SK}^{-1}(\delta)$ is associated
with standard kernel $K$ of order $-1$, and that $T1=0=T^t1$. Theorem
\ref{isolause} states that $T\in\mathrm{ADP}(\delta)$ and
Theorem \ref{kernel_theo} combined with
the kernel uniqueness, Proposition \ref{uniq},
implies that we have the identity
\[
K(x,y)=\sum_{Q \in \mathcal{D}} \sum_{P\in \mathcal{D}}
\langle T(\psi_P),\phi_Q\rangle \phi_P(y) \psi_Q(x)
\]
for every $(x,y)\in \R^n\times \R^n\setminus\{(x,x)\}$.
Using this representation it is possible to prove
different regularity results for the kernel $K$; as
an example one can obtain a homogeneity estimate for
the first order differences of the kernel $K$.
\end{rmk}

\section{A full $T1$ theorem via paraproducts}\label{kuussec}

Here the special $T1$ theorem is generalized so that
the assumption $T1=0=T^t1$ is
replaced by a weaker assumption $T1,T^t1\in \dot F^{12}_\infty$.
Under this weaker assumption the
operator $T\in\mathrm{SK}^{-1}(\delta)$ has a bounded extension
$\dot W^{\alpha, p}\to \dot W^{1+\alpha,p}$ for all $1<p<\infty$
and $-1\le \alpha\le 0$.

We also obtain a sharpness result stating
that if $T\in \mathrm{SK}^{-1}(\delta)$ is a potential operator
in the sense of Definition \ref{potopp},
then $T1,T^t1\in \dot F^{12}_\infty$.
Combining the results above, we find
that if $T\in\mathrm{SK}^{-1}(\delta)$ is a potential
operator, then it has 
a bounded extension
$\dot W^{\alpha, p}\to \dot W^{1+\alpha,p}$
for all $-1\le \alpha\le 0$ and $1<p<\infty$.

\subsection{A full $T1$ theorem with a converse}

The full $T1$ theorem is
obtained via reduction to Theorem \ref{isolause}
with the aid of  potential operators $\Pi_b$ and $\Pi^t_b$,
$b\in \dot F^{12}_\infty$, satisfying
\begin{itemize}
\item $\Pi_b 1=b$ and $\Pi_b^t1 =0$;
\item $\Pi_b$ and $\Pi_b^t$ are 
bounded $\dot F^{\alpha 2}_p\to \dot F^{1+\alpha,2}_p$ for all
  $1<p<\infty$ and $-1\le \alpha\le 0$;
\item $\Pi_b,\Pi_b^t\in \mathrm{SK}^{-1}(1)$.
\end{itemize}
The very end of this section is devoted to the construction of
such potential operators, known as {\em paraproduct operators}.
But assuming these properties for now, we can recover the full 
$T1$ theorem:

\begin{thm}\label{full}
Let $T\in \mathrm{SK}^{-1}(\delta)$ and
$T1,T^t1\in \dot F^{12}_\infty$.
Then $T$ has an
extension as a bounded operator 
$\dot F^{\alpha 2}_p\to \dot F^{1+\alpha,2}_p$
for all $1<p<\infty$ and $-1\le\alpha\le 0$.
\end{thm}

\begin{proof}
Let $a=T1$, $b=T^t1$, and 
$S = T-\Pi_a-\Pi^t_b$. Because $\Pi_a,\Pi^t_b\in \mathrm{SK}^{-1}(1)$,
we see that $S\in \mathrm{SK}^{-1}(\delta)$.
Moreover using the assumptions and the 
properties of paraproducts above, we have
for any $Q\in \mathcal{D}$
\[
\int_{\R^n} S(\psi_Q)(x){\rm{d}}x
=\lim_{j\to\infty} \langle \eta^j,S (\psi_Q)\rangle =
\lim_{j\to\infty} \langle S^t(\eta^j),\psi_Q\rangle
=0.
\]
In a similar way, we get
\[
\int_{\R^n} S^t(\phi_Q)(x){\rm{d}}x
=\lim_{j\to\infty} \langle S(\eta^j),\phi_Q\rangle
=0.
\]
Theorem \ref{isolause} implies that
$S:\dot F^{\alpha 2}_p\to \dot F^{1+\alpha,2}_p$ is a bounded
operator. Using the boundedness properties of $\Pi_b$, we
see that $T=S+\Pi_a+\Pi^t_b:\dot F^{\alpha 2}_p\to 
\dot F^{1+\alpha,2}_p$ 
is a bounded
operator.
\end{proof}

\begin{rmk}
Choose $b=\overline{\phi_{Q}}\in \dot F^{12}_\infty$.
Then $\langle b,\phi_Q\rangle = ||\phi||_2^2\not=0$
and therefore $\Pi_b 1\not=0$.
This shows the existence of a potential 
operator $\Pi_b$ not satisfying all the assumptions
of Theorem \ref{isolause}. 
\end{rmk}

Recall that the assumption $T1,T^t1\in {\mathrm{BMO}}$ is sharp in classical
$T1$ theorem \cite{D-J}. The following result describes a similar
sharpness result in our context:

\begin{thm}\label{t1conv}
Assume that $T\in\mathrm{SK}^{-1}(\delta)$ is a potential
operator. In other words, assume that it has a bounded extension
$\dot F^{\alpha 2}_2\to \dot F^{1+\alpha,2}_2$ for
all $-1\le\alpha\le 0$. Then $T1,T^t1\in \dot F^{12}_\infty$.
\end{thm}

\begin{proof}
We show that $T1\in \dot F^{12}_\infty$; the proof
for $T^t$ is completely analogous.
We first prove that the family $\{T(\eta^j)\}_{j\in \N}$ is 
bounded in  $\dot F^{12}_\infty$. Let $R=R_{-j,0}\in\mathcal{D}$; 
that is, $l(R)=2^j$ and $x_{R}=0$. Then 
$\phi(0)|R|^{-1/2}\eta^j=\phi_R$ and 
using Theorem \ref{jatkuv}, we have
\begin{eqnarray*}
||T(\phi_R)||_{\dot F^{12}_\infty}&\le& C||\{\langle T(\phi_R),
\phi_Q\rangle\}_{Q\in\mathcal{D}}||_{\dot f^{12}_\infty}\\
&=&C\sup_{P\in \mathcal{D}}
\bigg( \frac{1}{|P|} \sum_{Q\subset P} 
\big(|Q|^{-1/n}|\langle T(\phi_R),\phi_Q\rangle |\big)^2
\bigg)^{1/2}.
\end{eqnarray*}
Fix $P\in\mathcal{D}_\mu$; assume first that $|R|\le |P|$. Then
using the boundedness assumption and Theorem
\ref{jatkuv}, we obtain the estimate
\begin{eqnarray*}
&&\bigg(\frac{1}{|P|}\sum_{Q\subset P}
\big(|Q|^{-1/n}|
\langle T(\phi_R),\phi_Q\rangle|\big)^2\bigg)^{1/2}
\le ||
\{\langle T(|P|^{-1/2}\phi_R),
\phi_Q\rangle\}_{Q\in\mathcal{D}}||_{\dot f^{12}_2}\\
&&\le C||T(|P|^{-1/2}\phi_R)||_{\dot F^{12}_2}
\le C|P|^{-1/2}||\phi_R||_{\dot F^{02}_2} 
\le C|R|^{-1/2}.
\end{eqnarray*}
Assume then that $|P|\le |R|$ and denote 
$B=B(x_P,2\sqrt nl(P))$.
Using  Theorem \ref{jatkuv}, we have
\begin{eqnarray*}
&&\bigg(\frac{1}{|P|}\sum_{Q\subset P}
\big(|Q|^{-1/n}|\langle T(\phi_R),
\phi_Q\rangle|\big)^2\bigg)^{1/2}\\
&&=\bigg(\frac{1}{|P|}\sum_{Q\subset P}
\big(|Q|^{-1/n}|\langle T(\chi_B\phi_R+(1-\chi_B)\phi_R),
\phi_Q\rangle|\big)^2\bigg)^{1/2}\\
&&\le C_{\phi,\psi}|P|^{-1/2}
||T(\chi_B\phi_R)||_{\dot F^{12}_2}+\bigg(\frac{1}{|P|}
\sum_{Q\subset P}
\big(|Q|^{-1/n}|\langle (1-\chi_B)\phi_R,
T^t(\phi_Q)\rangle|\big)^2\bigg)^{1/2}.
\end{eqnarray*}
First term on the right hand side is easier to estimate;
using the 
boundedness
assumption for $T$ and the estimate 
$||\phi_R||_\infty\le C_\phi|R|^{-1/2}$, we get
\begin{eqnarray*}
|P|^{-1/2}||T(\chi_B\phi_R)||_{\dot F^{12}_2}
&\le& C_T|P|^{-1/2}||\chi_B\phi_R||_{\dot F^{02}_2}\\
&\le& C_{T,\phi}|P|^{-1/2}|R|^{-1/2}|P|^{1/2}=
C_{T,\phi}|R|^{-1/2}.
\end{eqnarray*}
Then we look at the second term. Fix a dyadic cube $Q$ with 
$Q\subset P$. 
Then 
using Lemma \ref{ar1}
and the estimate $||\phi_R||_\infty\le C_\phi|R|^{-1/2}$, we 
have
\begin{equation}\label{eski}
\begin{split}
&|\langle (1-\chi_B)\phi_R,T^t(\phi_Q)\rangle| \le 
\int_{\R^n\setminus B}
|\phi_R(x)||T^t(\phi_Q)(x)|{\rm d}x\\
&\le C_{T,\phi,\psi} |R|^{-1/2}l(Q)|Q|^{-1/2} 
\int_{\R^n\setminus B} 
\bigg(1+\frac{|x-x_Q|}{l(Q)}\bigg)^{-(n+\delta)}{\rm d}x.
\end{split}
\end{equation}
Note that if $x\in \R^n\setminus B$, then $|x-x_Q|\ge l(P)$ and 
therefore
$1+\frac{|x-x_Q|}{l(Q)}\ge \frac{l(P)}{l(Q)}$.
Using this we can estimate the right hand side of \eqref{eski} 
from above by
\begin{eqnarray*}
&&C_{T,\phi,\psi}|R|^{-1/2}l(Q)|Q|^{-1/2} 
l(Q)^{\delta/2}l(P)^{-\delta/2}\int_{\R^n} 
\bigg(1+\frac{|x-x_Q|}
{l(Q)}\bigg)^{-(n+\delta/2)}{\rm d}x\\
&&\le C_{T,\phi,\psi,\delta} 
|R|^{-1/2} |Q|^{1/2}l(Q)^{1+\delta/2}l(P)^{-\delta/2}.
\end{eqnarray*}
Denote $l(P)=2^{-\mu}$. Then  we get the following estimate for our 
second
term
\begin{eqnarray*}
&&\bigg(\frac{1}{|P|}\sum_{Q\subset P}
\big(|Q|^{-1/n}|\langle (1-\chi_B)\phi_R,
T^t(\phi_Q)\rangle|\big)^2\bigg)^{1/2}\\
&&\le 
C_{T,\phi,\psi,\delta}|R|^{-1/2}
l(P)^{-\delta/2}\bigg(\frac{1}{|P|}\sum_{Q\subset P}
|Q|l(Q)^{\delta}\bigg)^{1/2}\\
&&\le 
C_{T,\phi,\psi,\delta}|R|^{-1/2}
l(P)^{-\delta/2}\bigg(\frac{1}{|P|}\sum_{\nu=\mu}^\infty
\frac{|P|}{2^{-\nu n}}2^{-\nu n}2^{-\nu\delta}\bigg)^{1/2}\\
&&\le 
C_{T,\phi,\psi,\delta}
|R|^{-1/2}l(P)^{-\delta/2}\bigg(\sum_{\nu=\mu}^\infty
2^{-\nu\delta}\bigg)^{1/2}\\
&&\le C_{T,\phi,\psi,\delta}|R|^{-1/2}.
\end{eqnarray*}
Combining the uniform estimates  above, we see
that
\[
||T(\phi_R)||_{\dot F^{12}_\infty}\le C|R|^{-1/2}
\]
with $C$ independent of $R$. Multiplying both sides by 
$|R|^{1/2}/\phi(0)$, we see that
the family $\{T(\eta^j)\}_{j\in\N}$ is bounded in 
$\dot F^{12}_\infty$.

Recall that $(\dot F^{-1,2}_1)'\approx \dot F^{12}_\infty$;
for a proof, see \cite[pp. 79--80]{F-J1}. Denote
$X=\dot F^{-1,2}_1$. Then $X$ is separable
and Banach--Alaoglu theorem \cite[p. 70]{Rudin:1991} 
implies that there exists
a subsequence $\{\eta^{j_k}\}_{k\in\N}$ such that
$\lim_{k\to\infty}T(\eta^{j_k})= b\in \dot F^{12}_\infty$ in 
the weak* topology
of $\dot F^{12}_\infty$.
In particular, we have $\lim_{k\to\infty}\langle T(\eta^{j_k}),
\phi_Q\rangle = \langle b,\phi_Q\rangle$ for
all $Q\in\mathcal{D}$; on the other hand, we have for all 
$Q\in\mathcal{D}$
\[
\lim_{j\to\infty} \langle T(\eta^j),\phi_Q\rangle
=\int T^t(\phi_Q)=
\lim_{k\to\infty} \langle T(\eta^{j_k}),\phi_Q\rangle = 
\langle b,\phi_Q\rangle.
\]
The same conclusions hold true for $\psi_Q$'s 
and therefore  $T 1=b\in \dot F^{12}_\infty$.
\end{proof}

Recall that the $L^2$-boundedness of a Calder\'on--Zygmund 
operator implies its $L^p$-boundedness for $1<p<\infty$; the 
proof is based on a weak type $(1,1)$ estimate and interpolation
\cite[pp. 584--585]{grafakos}.
 Combining theorems \ref{full} and \ref{t1conv}
we immediately get the following analogous result:

\begin{cor}\label{k2}
Assume that $T\in\mathrm{SK}^{-1}(\delta)$ is a potential operator.
In other words, assume that it has a bounded extension
$\dot F^{\alpha 2}_2\to \dot F^{1+\alpha,2}_2$ for
all $-1\le\alpha\le 0$. 
Then
$T$ has a bounded
extension $\dot F^{\alpha 2}_p\to \dot F^{1+\alpha,2}_p$
for all $1<p<\infty$ and $-1\le\alpha\le 0$.
\end{cor}

\subsection{About the construction of paraproduct operators}
\label{viissec}

As seen already the paraproduct operators 
establish a reduction from the full $T1$ theorem
to the  special $T1$ theorem.
Here we construct these 
paraproduct operators
to complete the proof of the full $T1$ theorem,
Theorem \ref{full}.
The treatment here follows
closely \cite{Wang} and for the convenience of the
reader we provide some of the shorter details here.

\begin{defn}\label{dyam}
Let $\Phi\in\mathcal{S}$ be  such that 
${\mathrm{supp}} \hat \Phi\subset B(0,1)$
and $\hat \Phi(0)=1$.
For $b\in \dot F^{12}_\infty$ we define the paraproduct
operator $\Pi_b$ by
\begin{equation}\label{suppenee}
\Pi_b(f) = \sum_{Q\in\mathcal{D}} \langle b,\phi_Q\rangle 
|Q|^{-1/2}\langle f,\Phi_Q\rangle \psi_Q,
\end{equation}
where $f\in \mathcal{S}_0$.
\end{defn}

\begin{thm}\label{khom}
The series \eqref{suppenee} defining 
$\Pi_b(f)$ converges unconditionally in the weak* topology of
$\mathcal{S}'/\mathcal{P}$ and absolutely pointwise
(and these limits coincide).
Also, we have $\Pi_b\in \mathrm{SK}^{-1}(1)$.
\end{thm}

\begin{proof}
For the proof recall the notation from Section \ref{pkt}.
Denote $\pi_{Q,Q}=\langle b,\phi_Q\rangle |Q|^{-1/2}$
and $\pi_{P,Q}=0$ for every $P,Q\in\mathcal{D}$ such that 
$P\not=Q$.
It is easy to verify that $|\pi_{Q,Q}|\le C_{b,\phi}l(Q)$.
Therefore the following estimate is trivial for every
$P,Q\in\mathcal{D}$
\[
|\pi_{P,Q}| \le C_{b,\phi} \big(l(P)\wedge l(Q)\big) 
W_{P,Q}(3,3).
\]
Denote for $x\not=y$
\[
K(x,y):=\sum_{Q\in\mathcal{D}} \langle b,
\phi_Q\rangle |Q|^{-1/2}\Phi_Q(y)\psi_Q(x)=
\sum_{Q\in\mathcal{D}}
\sum_{P\in\mathcal{D}} \pi_{Q,P} \Phi_P(y)\psi_Q(x).
\]
Computing as in the proof  of Theorem \ref{kernel_theo} 
we see
that 
$K:\R^n\times \R^n\setminus \{(x,x)\}\to \C$ is 
continuous and
it satisfies properties \eqref{ee}--\eqref{kk} with 
$\delta=1$. 
Actually the same proof shows that, 
for any given $x\in \R^n$, we have the estimate
\begin{eqnarray*}
&&\sum_{Q\in\mathcal{D}} |\langle b,\phi_Q\rangle| |Q|^{-1/2} 
|\langle f,\Phi_Q\rangle| |\psi_Q(x)| \\
&&\le \int_{\R^n} \sum_{Q\in\mathcal{D}} 
|\langle b,\phi_Q\rangle| |Q|^{-1/2} |\Phi_Q(y)| |\psi_Q(x)||f(y)|
{\rm d}y\\
&&\le C\int_{\R^n}Ê\frac{|f(y)|}{|x-y|^{n-1}}{\rm d}y\le C||f||_\infty^{(n-1)/n}||f||_1^{1/n}<\infty;
\end{eqnarray*}
the second last inequality follows from the maximal inequality  
\eqref{maximal}.
Using dominated
convergence we see that $\Pi_b$ is associated with kernel $K$. 
Also we  see that the series \eqref{suppenee} converges 
absolutely pointwise
and unconditionally in $\mathcal{S}'/\mathcal{P}$.
\end{proof}

\begin{thm}\label{vastine}
Let $b\in \dot F^{12}_\infty$. Then
 $\Pi_b 1=b$ and $\Pi_b^t 1=0$.
\end{thm}

\begin{proof}
Recall that $\eta^j(x)=\phi(x/2^j)/\phi(0)$ and fix 
$P\in \mathcal{D}_\mu$.
Then we use the inequality
\[
\sup_{Q\in \mathcal{D}} |Q|^{-1/2}|\langle \eta^j,\Phi_Q\rangle|
<\infty,\]
both conclusions of Theorem \ref{phi_ide},
and Theorem \ref{khom}
in order to obtain
\begin{eqnarray*}
\lim_{j\to\infty} \langle \Pi_b(\eta^j),\phi_P\rangle
&=&\lim_{j\to\infty} \sum_{\nu=\mu-1}^{\mu+1}
\sum_{k\in \Z^n} \langle b,\phi_{Q_{\nu k}}
\rangle |Q_{\nu k}|^{-1/2}\langle \eta^j,
\Phi_{Q_{\nu k}}\rangle \langle \psi_{Q_{\nu k}},
\phi_P\rangle\\
&=&\sum_{\nu=\mu-1}^{\mu+1}\sum_{k\in \Z^n} \langle b,
\phi_{Q_{\nu k}}\rangle \langle \psi_{Q_{\nu k}},
\phi_P\rangle=\langle b,\phi_P\rangle.
\end{eqnarray*}
On the other hand, we have 
\[
\lim_{j\to\infty} \langle \Pi_b^t(\eta^j),
\phi_P\rangle
=\lim_{j\to\infty} \sum_{\nu\ge \mu-3}\sum_{k\in \Z^n} \langle b,
\phi_{Q_{\nu k}}\rangle |Q_{\nu k}|^{-1/2}\langle \phi_P,
\Phi_{Q_{\nu k}}\rangle \langle \psi_{Q_{\nu k}},
\eta^j\rangle=0
\]
since $\langle \psi_{Q_{\nu k}},\eta^j\rangle = 0$ for 
$j\ge 5-\mu$ and $\nu \ge \mu-3$. These combined with similar 
computations
for $\psi$ imply the desired conclusion.
\end{proof}

The $\dot F^{\alpha 2}_p\to \dot F^{1+\alpha,2}_p$ 
boundedness of 
the paraproduct can be  reduced to the  boundedness of certain matrix 
on the corresponding sequence
spaces. This matrix factors as a product of two matrices 
$T_c$ and $G$, whose boundedness properties can then be 
established separately. 
This program is carried out in detail by \cite{Wang},
see
Corollary 3.3., Lemma 3.4., and the proof of
Theorem 4.1.(a) therein.

\begin{thm}\label{paraj2}
Let $b\in \dot F^{12}_\infty$.
If $1< p<\infty$ and $\alpha\le 0$, then
the paraproduct operator $\Pi_b$ has a bounded
extension $\dot F^{\alpha 2}_p\to \dot F^{1+\alpha,2}_p$.
\end{thm}

%\begin{proof}
%Let $f\in \mathcal{S}_0$. Using Theorem \ref{phi_ide} with
%$\phi$ 
%and $\psi$ interchanged, we get
%\[
%\Pi_b(f) = \sum_{Q\in\mathcal{D}} \langle b,\phi_Q
%\rangle |Q|^{-1/2} \sum_{P\in\mathcal{D}} \langle f,
%\phi_P\rangle \langle \psi_P,\Phi_Q\rangle \Psi_Q.
%\]
%Denote $\gamma_Q=\langle b,\phi_Q\rangle$, 
%$\gamma=\{\gamma_Q\}_Q$,
%$s_P=\langle f,\phi_P\rangle$, and $s=\{s_P\}_P$.
%Using Theorem  \ref{jatkuv},
%Lemma \ref{aalem}, and Lemma \ref{beelem}, we obtain
%\begin{eqnarray*}
%||\Pi_b(f)||_{\dot F^{1+\alpha,2}_p} &\le& C\bigg|
%\bigg|\bigg\{\gamma_Q |Q|^{-1/2} \sum_P s_P\langle \psi_P,\Phi_Q
%\rangle\bigg\}_{Q\in\mathcal{D}}
%\bigg|\bigg|_{\dot f^{1+\alpha,2}_p}\\
%%&=&C||T_\gamma Gs||_{\dot f^{1+\alpha,2}_p}
%\le C||Gs||_{\dot f^{\alpha \infty}_p}
%\le C||s||_{\dot f^{\alpha 2}_p}
%\le C||f||_{\dot F^{\alpha 2}_p}.
%\end{eqnarray*}
%Note that if $\alpha<0$, we also used the
%continuity of the inclusion 
%$\dot f^{\alpha 2}_p\to \dot f^{\alpha \infty}_p$.
%\end{proof}

%Using Theorem \ref{paraj2} above and the duality mentioned in 
%the proof
%of Lemma \ref{aalem}, we obtain the following boundedness
%result for the transpose of the paraproduct.

Due to \cite[pp. 76--79]{F-J1} we have
$(\dot f^{\beta q}_p)' \approx \dot f^{-\beta q'}_{p'}$ for
$\beta\in\R$, $1\le p<\infty$, and $1\le q<\infty$.  This duality
combined with the previous theorem gives us:

\begin{cor}
Let $b\in \dot F^{12}_\infty$.
If $1<p<\infty$ and $\beta\ge -1$,
then the transposed paraproduct operator $\Pi_b^t$ has a bounded
extension $\dot F^{\beta 2}_p\to \dot F^{1+\beta,2}_p$.
\end{cor}

\section*{Acknowledgements}

During the time when research was conducted, the author was supported
by the Finnish Academy of Science and Letters, Vilho, Yrj\"o 
and  Kalle V\"ais\"al\"a Foundation
and by the Center of Excellence of Geometric Analysis 
and Mathematical Physics at the University of Helsinki. The author
wishes to thank Prof. Kari Astala, Prof. Tadeusz Iwaniec, and Dr. Tuomas Hyt\"onen
for advice.

\end{document}